\documentclass[11pt]{article} 
\usepackage{amsfonts,amsmath,latexsym,amssymb,mathrsfs,amsthm,comment}
\usepackage{graphicx}
\usepackage{xcolor}
\usepackage{booktabs}

\let\OLDthebibliography\thebibliography
\renewcommand\thebibliography[1]{
  \OLDthebibliography{#1}
  \setlength{\parskip}{1pt}
  \setlength{\itemsep}{0pt plus 0.0ex}
}

\def\numberlikeadb{\global\def\theequation{\thesection.\arabic{equation}}}
\numberlikeadb
\newtheorem{theorem}{Theorem}[section]
\newtheorem{lemma}[theorem]{Lemma}
\newtheorem{corollary}[theorem]{Corollary}

\newtheorem{remark}[theorem]{Remark}

\usepackage[multiple]{footmisc}

\usepackage{lscape}
\usepackage{caption}
\usepackage{multirow}
\allowdisplaybreaks
\begin{document}

\title{On P\'olya's random walk constants\footnote{Dedicated to the 130th anniversary of Lauricella functions.}}

\author{Robert E. Gaunt\footnote{Department of Mathematics, The University of Manchester, Oxford Road, Manchester M13 9PL, UK, robert.gaunt@manchester.ac.uk; Saralees.Nadarajah@manchester.ac.uk},\: Saralees Nadaraja$\mathrm{h}^{*}$\ and Tibor K. Pog\'any\footnote{Institute of Applied Mathematics, \'Obuda University, 1034 Budapest, Hungary, pogany.tibor@nik.uni-obuda.hu}\footnote{Faculty of Maritime Studies, University of Rijeka, 51000 Rijeka, Croatia, tibor.poganj@uniri.hr}}

\date{} 
\maketitle

\vspace{-5mm}

\begin{abstract} 
A celebrated result in probability theory is that a simple symmetric random walk on the $d$-dimensional lattice $\mathbb{Z}^d$ is recurrent for $d=1,2$ and transient for $d\geq 3$. In this note, we derive a closed-form expression, in terms of the Lauricella function $F_C$, for the return probability for all $d\geq3$. Previously, a closed-form formula had only been available for $d=3$.
\end{abstract}

\noindent{{\bf{Keywords:}}} Random walk; return probability; P\'olya's random walk constants; Lauricella function; Watson's triple integrals; Laplace transform

\noindent{{{\bf{AMS 2010 Subject Classification:}}} Primary 60G50; 33C65}

\section{Introduction}

Let $p(d)$ be the probability that a simple symmetric random walk on the $d$-dimensional lattice $\mathbb{Z}^d$ returns to origin, for $d\geq1$. A celebrated result of P\'olya \cite{p21} states that $p(1)=p(2)=1$ but $p(d)<1$ for $d\geq3$. An explicit formula is available in the three-dimensional case:
\[p(3)=1-1/u(3)=0.3405373296\ldots,\]
where
\begin{align}
u(3)&=\frac{3}{(2\pi)^3}\int_{-\pi}^\pi\!\int_{-\pi}^\pi\!\int_{-\pi}^\pi\frac{ \mathrm{d}x\, \mathrm{d}y\, \mathrm{d}z}{3-\cos x-\cos y-\cos z} \label{wat}\\
&=\frac{\sqrt{6}}{32\pi^3}\Gamma\Big(\frac{1}{24}\Big)\Gamma\Big(\frac{5}{24}\Big)\Gamma\Big(\frac{7}{24}\Big)\Gamma\Big(\frac{11}{24}\Big) \label{gammas}\\
&=1.5163860592\ldots \nonumber
\end{align}
(see \cite{d54,gz77,mw40,watson}). The integral in (\ref{wat}) is one of Watson's triple integrals \cite{watson} up to a multiplicative factor.

Closed-form expressions for the case $d\geq4$ are not available to date in the literature, although numerical values are reported in \cite{f03,m56}, asymptotic expansions as $d\rightarrow\infty$ are given in \cite{j01}, and an integral representation was obtained by \cite{m56}: for $d\geq3$,
   \begin{equation} \label{FY}
	    p(d)=1-1/u(d),
	 \end{equation}
where
\begin{align}
u(d)&=\int_{(-\pi,\pi)^d} \bigg(d-\sum_{k=1}^d \cos x_k\bigg)^{-1} \, \mathrm{d}x_1\, \mathrm{d}x_2\cdots\, \mathrm{d}x_d \label{wat2}\\
&=\int_0^\infty \left[ I_0 \left( \frac {x}{d} \right) \right]^d  \mathrm{e}^{-x}  \, \mathrm{d}x, \label{const}
\end{align}
 with  $I_0 (\cdot)$ denoting the
modified Bessel function of the first kind of order zero, defined by
\begin{align}
I_0 (x) = \sum_{k \geq 0} \frac1{\left( k! \right)^2} \left( \frac {x}{2} \right)^{2k}.
\label{series}
\end{align}
The integral in (\ref{wat2}) is a $d$-fold integral generalisation of the Watson triple integral (\ref{wat}) (again, up to a 
multiplicative factor). Note that the integral (\ref{const}) is not convergent for $d=1,2$, which is easily seen from the limiting 
form $I_0(x)\sim \mathrm{e}^x/\sqrt{2\pi x}$, $x\rightarrow\infty$ (see \cite{olver}). 

In this note, we derive a closed-form expression for the return probability $p(d)$ for any positive integer $d\geq3$.
The expression involves the Lauricella function $F_C$ (see \cite{e78,l93}), defined by
\begin{align}
\displaystyle
F_C^{(d)} \left( a, b; c_1, \ldots, c_d; x_1, \ldots, x_d \right)  =
\sum_{k_1 \geq 0} \cdots \sum_{k_d \geq 0}
\frac {(a)_{k_1 + \cdots + k_d} (b)_{k_1 + \cdots + k_d}}
{\left( c_1 \right)_{k_1} \cdots \left( c_d \right)_{k_d}}
\frac {x_1^{k_1} \cdots x_d^{k_d}}
{k_1! \cdots k_d!},
\label{hyp}
\end{align}
where $(f)_k = f (f + 1) \cdots (f + k - 1) = \Gamma(f+k)/\Gamma(f)$ denotes the ascending factorial or the Pochhammer symbol. Numerical routines for the direct computation of (\ref{hyp}) are available; see, for instance, the \emph{Mathematica}-based  routine presented in \cite{Bytev}. 

\section{Closed-form expression for the return probability}

Our main result is the following.

\begin{theorem}
For any positive integer $d\geq3$,
   \begin{equation} \label{FX0}
      u(d) = F_C^{(d)} \left( 1, \frac12; 1, \ldots, 1; \frac {1}{d^2}, \ldots, \frac {1}{d^2} \right).
   \end{equation}
\end{theorem} 

\begin{proof}
Using (\ref{series}), we can write (\ref{const}) as
\begin{align}
\displaystyle
u (d)
&=
\displaystyle
\int_0^\infty \left[ \sum_{k \geq 0} \frac1{\left( k! \right)^2} \left( \frac {x}{2 d} \right)^{2k} \right]^d  \mathrm{e}^{-x} \,  \mathrm{d}x
\nonumber
\\
&=
\displaystyle
\int_0^\infty \sum_{k_1 \geq 0} \cdots \sum_{k_d \geq 0}
\frac {1}{\left( k_1! \cdots k_d! \right)^2} \left( \frac {x}{2 d} \right)^{2k_1 + \cdots + 2k_d}   \mathrm{e}^{-x} \,  \mathrm{d}x
\nonumber
\\
&=
\displaystyle
\sum_{k_1 \geq 0} \cdots \sum_{k_d \geq 0}
\frac {1}{\left( k_1! \cdots k_d! \right)^2 \left( 2 d \right)^{2k_1 + \cdots + 2k_d}}
\int_0^\infty  x^{2k_1 + \cdots + 2k_d}   \mathrm{e}^{-x}  \, \mathrm{d}x
\nonumber
\\
&=
\displaystyle
\sum_{k_1 \geq 0} \cdots \sum_{k_d \geq 0}
\frac {1}{\left( k_1! \cdots k_d! \right)^2 \left( 2 d \right)^{2k_1 + \cdots + 2k_d}}
\Gamma\left( 2k_1 + \cdots + 2k_d + 1 \right).
\label{tt}
\end{align}
Using the duplication formula for the gamma function, (\ref{tt}) can be written  as
\begin{align}
\displaystyle
u (d)
&=
\displaystyle
\frac1{\sqrt{\pi}} \sum_{k_1 \geq 0} \cdots \sum_{k_d \geq 0}
\frac {1}{\left( k_1! \cdots k_d! \right)^2 d^{2k_1 + \cdots + 2k_d}}
\Gamma\left( k_1 + \cdots + k_d + \frac  {1}{2} \right)
\Gamma\left( k_1 + \cdots + k_d + 1 \right)
\nonumber
\\
&=
\displaystyle
\sum_{k_1 \geq 0} \cdots \sum_{k_d \geq 0}
\frac {(1)_{k_1 + \cdots + k_d}  \left( \frac {1}{2} \right)_{k_1 + \cdots + k_d}}
{(1)_{k_1} \cdots (1)_{k_d} k_1! \cdots k_d! d^{2k_1 + \cdots + 2k_d}}.
\nonumber
\end{align}
Now \eqref{FX0} follows from the definition in \eqref{hyp}.
\end{proof} 

\begin{remark}
The return probability \eqref{FY} becomes
   \[ p(d) = 1-\bigg[F_C^{(d)} \left( 1, \frac12; 1, \ldots, 1; \frac {1}{d^2}, \ldots, \frac {1}{d^2} \right)\bigg]^{-1}\,,\]
for all positive integers $d\geq3$.
\end{remark}

\begin{remark} By expressing the return probability $p(d)$ for any positive integer $d\geq3$ in terms of the Lauricella function $F_C$, which is a well-studied special function with in-built numerical routines for direct computation, following standard conventions within the special functions literature, our formula {\rm{(\ref{FX0})}} can be endowed with the label ``closed-form," as has been done in works such as {\rm{\cite{ieee}}}. This is in contrast to the integral representations {\rm{(\ref{wat2})}} and {\rm{(\ref{const})}}, which until now have not been evaluated in terms of known special functions. It is common practice to evaluate integrals in terms of Lauricella functions; see, for example, the standard reference {\rm{\cite{Prudnikov4}}} (a corrected example of such a formula is given in Lemma {\rm{\ref{lem1}}}).
\end{remark}

\begin{corollary}The following reduction formula holds:
    \begin{align}
        F_C^{(3)} \left( 1, \frac12; 1, 1, 1; \frac {1}{9}, \frac{1}{9}, \frac {1}{9} \right)=\frac{\sqrt{6}}{32\pi^3}\Gamma\Big(\frac{1}{24}\Big)\Gamma\Big(\frac{5}{24}\Big)\Gamma\Big(\frac{7}{24}\Big)\Gamma\Big(\frac{11}{24}\Big). \label{red}
    \end{align}
\end{corollary}

\begin{proof}
    Combine (\ref{gammas}) and (\ref{FX0}).
\end{proof}

\begin{remark}
{\rm 1.} The reduction formula \eqref{red} appears to be new. We could not locate it in standard references such as 
{\rm \cite{sk85}}.    

\vspace{2mm}

\noindent {\rm 2.} We were unable to obtain a further simplification of \eqref{FX0} for $d\geq4$ from reduction formulas 
for Lauricella functions in standard references such as {\rm \cite{sk85}}. However, we cannot not rule out this possibility,  
especially in the light of  the fact that we could not locate  \eqref{red} in the existing literature.
\end{remark}

The direct Laplace transform \cite[p. 346, Eq. 3.15.16.35]{Prudnikov4} turns out to be erroneous. Here we give its corrected form. On specifying $\lambda=\nu_j=0$, $a_j = d^{-1}$ and $p=1$ in \eqref{FX1} below we arrive at \eqref{FX0}. Recall that the modified Bessel function of the first kind of order $\nu\in\mathbb{R}$ is defined for $x\in\mathbb{R}$ by the power series
\begin{align}
I_ \nu(x) = \sum_{k \geq 0} \frac1{ k! \Gamma(\nu+k+1)} \left( \frac {x}{2} \right)^{\nu+2k}. \label{series2}
\end{align}

\begin{lemma}\label{lem1} Denote $\nu = \sum_{j=1}^d \nu_j$, where $d$ is a positive integer. Let $\Re(\lambda) + \nu>-1$ and $\nu_j>-1$; $j = 1,\ldots,d$. Let $a_1,\ldots,a_d>0$. 
Then, the Laplace transform 
   \begin{align} \label{FX1}
	    \mathscr L_p\Big[ x^\lambda \prod\limits_{j=1}^d I_{\nu_j}(a_j x)\Big]\, {\rm d}x 
			     &= \frac{\Gamma(\lambda + \nu +1)}{2^\nu p^{\lambda+\nu+1}} \bigg\{\prod_{j=1}^d \frac{a_j^{\nu_j}}{\Gamma(\nu_j+1)}\bigg\} \notag \\
			     &\quad \cdot F_C^{(d)} \Big( \frac{\lambda+\nu+1}2, \frac{\lambda+\nu}2+1; \nu_1+1, 
			        \dots,\nu_d+1; \frac{a_1^2}{p^2}, \dots, \frac{a_d^2}{p^2} \Big)\,,
	 \end{align}
provided $p>\sum_{j=1}^d a_j$, or $p=\sum_{j=1}^d a_j$ and $\Re(\lambda)<d/2-1$.
\end{lemma}  

\begin{proof} The conditions $p>\sum_{j=1}^d a_j$, or $p=\sum_{j=1}^d a_j$ and $\Re(\lambda)<d/2-1$ are required to ensure that the integral in the Laplace transform is convergent; this is easily seen from the limiting form $I_\nu(x)\sim \mathrm{e}^x/\sqrt{2\pi x}$, $x\rightarrow\infty$ (see \cite{olver}).

Applying the power series definition (\ref{series2}) of the function $I_\nu(x)$, denoting $\boldsymbol n = 
(n_1,\ldots,n_d)$ and $n = \sum_{j=1}^d n_j$, we conclude by the Legendre duplication formula (twice) that
   \begin{align*}
      \mathscr L_p\Big[ x^\lambda \prod\limits_{j=1}^d & I_{\nu_j}(a_j x)\Big]\, {\rm d}x 
			      = \int_0^\infty {\rm e}^{-px} x^\lambda \prod_{j=1}^d I_{\nu_j}(a_j x)\,{\rm d}x \\
					 &= \sum_{\boldsymbol n \geq 0} \prod_{j=1}^d \frac{\left(\frac{a_j}2\right)^{2n_j+\nu_j}}{\Gamma(n_j+\nu_j+1)\,n_j!} 
					    \int_0^\infty {\rm e}^{-p x} x^{\lambda+2n+\nu+1} \,{\rm d}x \\
					 &= \sum_{\boldsymbol n \geq 0} \prod_{j=1}^d \frac{\left(\frac{a_j}2\right)^{2n_j+\nu_j}}{\Gamma(n_j+\nu_j+1)\,n_j!} 
					    \frac{\Gamma(\lambda + 2n + \nu + 1)}{p^{\lambda + 2n + \nu + 1}} \\
					 &= \frac{2^\lambda}{\sqrt{\pi} p^{\lambda+\nu+1}} \prod_{j=1}^d \frac{a_j^{\nu_j}}{\Gamma(\nu_j+1)} 
					    \sum_{\boldsymbol n \geq 0} \Gamma\Big(\frac{\lambda+\nu+1}2+ n\Big) 
							\Gamma\Big( \frac{\lambda+\nu}2+1+ n\Big) \prod_{j=1}^d \frac{(a_j^2/p^2)^{n_j}}
							{(\nu_j+1)_{n_j}\,n_j!} \\
					 &= \frac{\Gamma(\lambda+\nu+1)}{2^\nu p^{\lambda+\nu+1}} \prod_{j=1}^d \frac{a_j^{\nu_j}}{\Gamma(\nu_j+1)} 
					    \sum_{\boldsymbol n \geq 0} \left(\frac{\lambda+\nu+1}2\right)_n \left(\frac{\lambda+\nu}2+1\right)_n 
							\prod_{j=1}^d \frac{(a_j^2/p^2)^{n_j}}{(\nu_j+1)_{n_j}\,n_j!} ,  
	 \end{align*}
which is equivalent to the statement. 
\end{proof} 

\section*{Acknowledgements} We would like to thank the referee for their constructive comments.

\end{document}